\documentclass[12pt]{article}
\usepackage[letterpaper, left=1.2in,right=1.2in]{geometry}

\usepackage{amsmath, amsfonts, amssymb, amsthm}
\usepackage{verbatim, float, caption, subcaption, graphicx, epstopdf, enumitem, setspace}
\usepackage[colorlinks, citecolor=blue, linkcolor=blue]{hyperref}

\usepackage{natbib}
\theoremstyle{plain}
\newtheorem{theorem}{Theorem}

\newtheorem{lemma}[theorem]{Lemma}

\theoremstyle{remark}

\newcommand{\norm}[1]{\left\lVert#1\right\rVert}

\newcommand{\e}{\epsilon}

\newcommand{\R}{\mathbb{R}}
\newcommand{\Z}{\mathbb{Z}}
\newcommand{\Q}{\mathcal{Q}}

\newcommand{\E}{\mathbb{E}}
\newcommand{\Prob}{\mathbb{P}}
\renewcommand{\P}{\mathcal{P}}
\newcommand{\Var}{\text{Var}}

\renewcommand{\H}{\mathcal{H}}

\newcommand{\X}{\textbf{X}}

\newcommand{\K}{\mathcal{K}}

\let\temp\phi
\let\phi\varphi
\let\varphi\temp

\begin{document}

\title{Convergence of Independence Metropolis-Hastings Algorithms in High Dimensions: Iteration Lengths, Cutoff Profiles, and Diagnostic Unreliability}
\author{Austin Brown, Department of Statistics, \\ Texas A\&M University, TX, USA }
\maketitle

\begin{abstract} 
We characterize the high-dimensional convergence of the arithmetic mean from the independence Metropolis-Hastings algorithm initialized at its proposal. Convergence is measured in the worst-case expected absolute error over bounded functions. We prove necessary and sufficient conditions for the mean to converge in relation to the importance weight. We precisely describe the logarithmic cutoff profile of the sharp decay to convergence where the relative entropy between the target and proposal determines the required iteration length. These results make connections between the required iterations of independence Metropolis-Hastings and importance sampling. For concentrating Bayesian posteriors, we illustrate how the required iteration time can be estimated providing concrete guidance for simulations. We justify the logarithmic cutoff analysis by showing a standard cutoff phenomenon is absent and prove that popular convergence diagnostics based on simulation output can validate inaccurate simulations. 
\end {abstract}



\section{Introduction}

Metropolis-Hastings algorithms \citep{hastings:1998, tierney:1998} are essential for generating samples to accurately estimate many high-dimensional integrals in statistics.
The independence Metropolis-Hastings (IMH) algorithm is a simple version, yet still widely used as a final correction step for proposals generated from variational inference or generative artificial intelligence. 
Every proposal is drawn from a fixed distribution, independently of the current state, and a move from the previous state to the new state is accepted or rejected.
Simulating the IMH algorithm does not require evaluating the normalizing constant of the target or proposal measures and unlike in general Metropolis-Hastings, the proposals and importance weights can be generated in parallel for the IMH algorithm.
Theory developed for this algorithm is often a foundation for understanding more sophisticated Metropolis-Hastings algorithms.

This work is motivated by and draws strong parallels to the required number of samples for importance sampling in high dimensions \citep{chatterjee:diaconic:2018}.
Unlike the requirements in importance sampling \citep{chatterjee:diaconic:2018, agapiou:omiros:2017, SanzAlonso}, the required number of samples for an accurate simulation in general is not understood in high dimensions.
There are many key differences between the IMH algorithm and importance sampling as the IMH algorithm simulates a dependent Markov chain and the arithmetic mean is used to estimate integrals instead of a weighted mean.
The IMH algorithm has many interesting theoretical properties lacking in importance sampling such as convergence in total variation which is not available for importance sampling directly.

We characterize precisely the number of iterations for the high-dimensional $L^1$ convergence of the arithmetic mean of the IMH algorithm initialized at the proposal.
The IMH algorithm must reduce the initialization bias and explore enough points for the arithmetic mean to reduce the $L^1$ error and approximate an integral well.
We show an abrupt transition in the logarithmic number of iterations or a logarithmic cutoff phenomenon \citep{aldous:diaconis:1986, chatterjee:diaconic:2018}.
High dimensions here can mean a large state space dimension or large data set size.
This is useful for studying high-dimensional probability measures or concentration of posteriors for large sample sizes in Bayesian statistics.

Let $\Z_+$ denote the positive integers.
Let $n \in \Z_+$ be an index parameter and let $\Pi_n$ be a target probability measure on a Borel state space $\X_n$.
Let $\Q_n$ be a proposal probability measure on $\X_n$ and $T_n \in \Z_+$ be an iteration time.
Let $X_{n}(0) \sim \Q_n$ be initialized at the proposal, and $X_{n}(0), \ldots, X_{n}(T_n - 1)$ be samples from an independence Metropolis-Hastings (IMH) process.
We study the precise convergence behavior of the $L^1$ loss as $n \to \infty$
\begin{align}
R_{T_n}( \Q_n, \Pi_n) = \sup_{ \norm{f}_{\infty} \le 1 } \E\left( \left| \frac{1}{T_n} \sum_{ t = 0 }^{T_n-1} f(X_{n}(t)) - \int_{\X_n} f d\Pi_n \right| \right)
\label{eq:risk}
\end{align}
where the supremum is taken over Borel measurable functions $f : \X_n \to [-1, 1]$. 

The contributions of this article are summarized as follows.
Section~\ref{section:main_results} provides a complete high-dimensional convergence profile of the risk \eqref{eq:risk}.
Let $w_n$ be an indexed importance weight and we prove that if $T_n \to \infty$, then $w_n / T_n \to 0$ in $\Pi_n$ probability is necessary and sufficient for $R_{T_n}( \Q_n, \Pi_n) \to 0$.
Section~\ref{section:main_results} also proves an abrupt transition towards convergence of \eqref{eq:risk} through a logarithmic cutoff and cutoff profile.
Roughly, if the relative entropy $\E_{\Pi_n}\left( \log(w_n) \right) \to \infty$ and $\log(w_n) / \E_{\Pi_n}\left( \log(w_n) \right) \to 1$, we prove that $\log(T_n) \approx \E_{\Pi_n}\left( \log(w_n) \right)$ is the logarithmic required number of iterations for \eqref{eq:risk} to be accurate.
Section~\ref{section:applications} illustrates the theory on concentrating Bayesian posteriors in Bayesian logistic regression with Laplace proposals where the required number of iterations can be estimated concretely.
Section~\ref{section:failure_of_total_variation_cutoff} proves that a standard cutoff to equilibrium in total variation is not possible for the IMH process.
Section~\ref{section:diagnostic_failure} provides further motivation for the theory by showing unreliability of convergence diagnostics such as the R-hat \citep{vehtari:etal:2021} for the IMH process.
We prove convergence diagnostics can provide false positives and fail to be a reliable means of verifying convergence for the IMH algorithm.

The convergence behavior of the IMH algorithm has previously been studied often with  bounded importance weights \citep{tierney:1994}.
Previous results address eigenvalues on finite state spaces \citep{Liu} and hitting times \citep{MaciucaZhu}.
A quantitative characterization of geometric convergence from any initialization is known for bounded importance weights \citep{MengersenTweedie}. 
The exact convergence from fixed point starts in both total variation and Wasserstein distance has been developed with bounded importance weights \citep{Wang, BrownJones}. 
More recent results address unbounded importance weights \citep{DeligiannidisEtAl} and \citep{DaeijavadAsoodeh}.  
Other results instead consider asymptotic-variance criteria \citep{DeligiannidisLee} or study large deviations of the empirical measure \citep{MilinanniNyquist}. 
There is related but different work on continuous-time Metropolis generators and high-dimensional
independence-sampler jump limits \citep{ChoiHuang, LeeNeal}.

\section{Background and Motivation}
\label{section:background}

We consider a sequence of target measures indexed by some parameter such as the dimension or data size.
Let $(\X_n)_{n \ge 0}$ be a sequence of Borel measurable spaces equipped with their Borel sigma-algebra.
For each $n \in \Z_+$, we are interested in estimating integrals from a Borel target probability measure $\Pi_n$ using a proposal Borel probability measure $\Q_n$ on $\X_n$.
We will assume we may choose an importance weight $w_n = d\Pi_n / d\Q_n$ that is everywhere finite on $\X_n$, which is possible when $\Pi_n$ is absolutely continuous with respect to $\Q_n$.

Define the acceptance function $\alpha_n : \X_n \times \X_n \to [0, 1]$ by
\[
\alpha_n(x, y) = 1 \wedge \frac{w_n(y)}{w_n(x)}
\]
if $w_n(x) > 0$ and $1$ otherwise.
Denote the Dirac measure at $x$ by $\delta_{x}$.
Define the independence Metropolis-Hastings kernel for any $x \in \X_n$ and Borel $B \subseteq \X_n$ by
\[
\P_n(x, B) = 
\int_B \alpha_n(x, y) \Q_n(dy) + r_n(x) \delta_{x}(B)
\]
where the rejection probability is $r_n(x) = 1 - \int_{\X_n} \alpha_n(x, y) \Q_n(dy)$.
We then define $\P_n^0(x, B) = \delta_x(B)$ and for $t \ge 1$, we define $\P^t_n(x, B) = \int \P_n(y, B)  \P^{t - 1}_n(x, dy)$.

The marginal distribution of $X_n(t)$ is denoted by $\Q_n \P_n^t$.
With $X_{n}(0) \sim \Q_n$, this Markov kernel defines an IMH process as a discrete-time Markov process $(X_{n}(t))_{t \ge 0}$.
We will only consider the proposal start process throughout the manuscript.
We may view $\Q_n$ as the best approximation available to $\Pi_n$ and this is then the optimal initialization available in practice and the initialization we consider.

\section{High-dimensional Convergence Theory and Cutoff Profiles}
\label{section:main_results}

Define the $\e$-approximation time of a risk for any $\e \in (0, 2)$
\[
T_{\e}(\Q_n, \Pi_n) = \inf\{ T \in \Z_+ :  R_T(\Q_n, \Pi_n) \le \e \}
\]
and $\inf \varnothing = \infty$.
Our goal is to characterize the sharp decrease in the risk \eqref{eq:risk} visualized in Figure~\ref{fig:diagram}.
The goal here is to characterize the required number of samples for Metropolis-Hastings in high dimensions.
This is a sharp phase transition called cutoff.

\begin{figure}[t]
\centering
\includegraphics[width=0.8\textwidth]{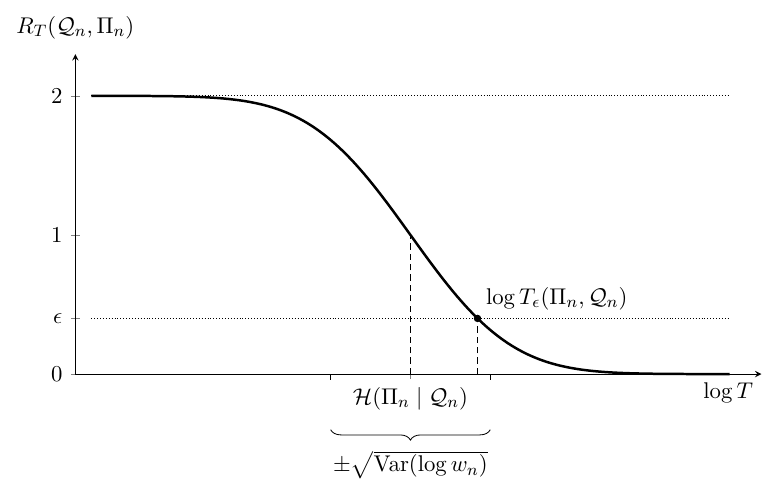}
\caption{A visualization of the cutoff for the IMH process. As $\log T$ passes the relative
entropy $\H(\Pi_n, \Q_n)$, the risk $R_T(\Q_n,\Pi_n)$ falls from $2$ to $0$ along
the profile $2\{1-\Phi(z)\}$ of Theorem~\ref{theorem:cutoff_result}.}
\label{fig:diagram}
\end{figure}

We have the following sharp characterization of high-dimensional convergence in terms of the importance weight.

\begin{theorem}
\label{theorem:convergence_characterization}
Suppose the iteration length $T_n \to \infty$.
Then,
\[
\lim_{n \to \infty} \Pi_n\left( \frac{w_n}{T_n}  \le \delta \right) = 1
\]
for every $\delta \in (0, 1)$ is equivalent to
\[
R_{T_{n}}(\Q_{n}, \Pi_{n}) \to 0.
\]
\end{theorem}

We can characterize more precisely the sharp decrease of the risk in terms of the relative entropy.
The main investigation is how the log weight concentrates towards the relative entropy $\H(\Pi_n, \Q_n) = \int\log(w_n) d\Pi_n$ so that in high dimensions,
\[
\log\left( w_n \right) \approx \H(\Pi_n, \Q_n).
\]
Let $N(0, 1)$ denote the standard Gaussian distribution and $\Phi$ be the distribution function of a standard Gaussian random variable.
We will characterize the logarithmic cutoff profile through the asymptotic normality of
\begin{align}
Z_n = \frac{\log(w_n) - \H( \Pi_n, \Q_n)}
{ \sqrt{ \Var_{\Pi_n}( \log(w_n) ) } }.
\label{eq:Z_n}
\end{align}

\begin{theorem}
\label{theorem:cutoff_result}
For $n \in\Z_+$, assume $\H(\Pi_n, \Q_n) < \infty$ and $\lim_{n \to \infty} \H(\Pi_n, \Q_n)) = \infty$.
If for any $\delta \in (0, 2)$,
\[
\lim_{n \to \infty} \Pi_n\left( \left| \frac{ \log(w_n)}{ \H(\Pi_n, \Q_n) } - 1 \right| \le \delta \right)
= 1,
\]
then for any $\e \in (0, 2)$,
\begin{align}
\lim_{n \to \infty} \frac{ \log( T_{\e}(\Q_n, \Pi_n) ) }{ \H(\Pi_n, \Q_n)  } = 1.
\label{eq:cutoff}
\end{align}
Assume $0 < \Var_{\Pi_n}( \log(w_n) ) < \infty$ and $\Var_{\Pi_n}( \log(w_n) ) \to \infty$ and
\[
Z_n \to N(0, 1)
\]
weakly in $\Pi_n$.
For $z \in \R$, let $T_n(z) = 1 \vee \lfloor \exp( \H( \Pi_n, \Q_n) + z \sqrt{ \Var_{\Pi_n}( \log(w_n) ) } ) \rfloor$ be an iteration time. 
Then
\begin{align}
\lim_{n \to \infty} R_{T_n(z)}(\Q_n, \Pi_n) = 2 (1 - \Phi(z)).
\label{eq:cutoff_profile}
\end{align}
\end{theorem}

The second conclusion describes the phase transition in terms of the variance of the log weight or the varentropy which has more recently shown up in Markov process cutoff theory \citep{salez:2026}.
We develop the proofs of Theorem~\ref{theorem:convergence_characterization} and Theorem~\ref{theorem:cutoff_result}.
We first prove a general upper and lower bound on the risk.

\begin{theorem}
\label{theorem:upper_lower_bound}
For any Borel probability measures $\Pi, \Q$ on a Borel space $\X$ with importance weight $w = d\Pi / d\Q$ that is everywhere finite on $\X$, and for all $T \ge 1$,
\[
R_T(\Q, \Pi) 
\ge 2 \sup_{B \subseteq \X, B \text{Borel}} \left\{ \Pi(B) \left( 1 - Q(B) \right)^{T} \right\}.
\] 
Let $u > 0$ and $C = \{ w \le u \}$ such that $\Pi(C) > 0$ and let $\lambda_C = 1 - \Q(C) + \Pi(C) / u$.
Then
\begin{align*}
R_T(\Q, \Pi)
\le 
\sqrt{ \frac{2}{ T \lambda_C } } + \frac{2 \Q(C) }{ T \lambda_C }
+ 2 \int_{C^c} T \wedge \frac{w}{\Pi(C)} d\Q
+ 2 ( 1 - \Pi(C) ).
\end{align*}
\end{theorem}

\begin{proof}[Proof of Theorem~\ref{theorem:upper_lower_bound}]
The problem that yields the lower bound is a tail distribution mismatch.
Let $B \subseteq \X$ be Borel.
Then we have
\[
\sup_{x \in B^c} \P(x, B) \le Q(B)
\]
and therefore,
\[
\inf_{x \in B^c} \P(x, B^c) \ge 1 - Q(B).
\]
Let $B_T^c = \{ X_k \in B^c, 0 \le k \le T - 1 \}$ and let $\phi = I_B - I_{B^c}$. 
Then on $B_T^c$,
\[
\frac{1}{T}\sum_{k = 0}^{T-1} \phi(X_k)  - \int_{\X} \phi d\Pi
= -2 \Pi(B).
\]
This gives the lower bound
\begin{align*}
\E\left( \left| \frac{1}{T}\sum_{k = 0}^{T-1} \phi(X_k)  - \int_{\X} \phi d\Pi \right| \right)
&\ge \E\left( \left| \frac{1}{T}\sum_{k = 0}^{T-1} \phi(X_k)  - \int_{\X} \phi d\Pi \right| I_{B_T^c}\right)
\\
&\ge 2 \Pi(B) \E\left( I_{B_T^c}\right)
\\
&\ge 2 \Pi(B) \left( 1 - Q(B) \right)^{T}.
\end{align*}

We now prove the upper bound.
The technique here is essentially that of a drift and minorization argument.
In particular, IMH samples $\Pi_C = \Pi(\cdot \mid C)$ and this can be controlled by a standard convergence argument conditioned on this set and Lyapunov control off of this set.
Let $\phi : \X \to \R$ with $\norm{\phi}_{\infty} \le 1$ and let $\psi = \phi -  \int_{\X} \phi(x) \Pi_C(dx)$.
This gives the upper bound through the decomposition
\begin{align*}
&\left| \frac{1}{T} \sum_{ k = 0 }^{ T - 1 } \phi(X_k) - \int_{\X} \phi d\Pi \right|
\\
&\le \left| \frac{1}{T} \sum_{ k = 0 }^{ T - 1 } I_{C}(X_k) \psi(X_k) \right|
+ \left| \frac{1}{T} \sum_{ k = 0 }^{ T - 1 } \psi(X_k) I_{C^c}(X_k) \right|
+  \left| \int_{\X} \psi d\Pi \right|.
\end{align*}
The last term is controlled by
\[ 
\left| \int_{\X} \phi(x) \Pi_C(dx) - \int_{\X} \phi d\Pi \right|
\le 2 \Pi(C^c)
= 2 (1 - \Pi(C)).
\]

The first term is bounded by a fairly standard convergence argument with bias and variance control.
Let $\P_C$ denote IMH with proposal $\Q_C = \Q(\cdot \mid C)$ and target $\Pi_C$.
This restricted Markov kernel satisfies a standard global minorization.
For any Borel $B \subseteq \X$ and $x \in C$,
\[
\P_C(x, B)
\ge \int_{B \cap C} \frac{\alpha(x, y)}{\Q(C)} \Q(dy)
\ge \frac{1}{\Q(C)} \int_{B \cap C} \frac{w(y)}{u} \Q(dy)
= \frac{\Pi(C)}{\Q(C) u} \Pi_C( B ).
\]
Changing the measure we can see $\frac{\Pi(C)}{\Q(C) u} \le 1$.
We can write the restricted residual kernel for $x \in C$, if $\frac{\Pi(C)}{\Q(C) u} < 1$
\[
R_C(x, \cdot) = \frac{\P_C(x, \cdot) - \frac{\Pi(C)}{\Q(C) u} \Pi_C(\cdot)}{1 - \frac{\Pi(C)}{\Q(C) u}}
\]
and if $\frac{\Pi(C)}{\Q(C) u} = 1$, $R_C(x, \cdot) = \Pi_C$.
Let $\lambda_C = 1 - \Q(C) + \frac{\Pi(C)}{u}$.
Let 
\[
\nu(\cdot) = \left[ \Q(\cdot \cap C^c) + \Pi(C) \Pi_C(\cdot) / u \right] / \lambda_C.
\]
On $C^c$, the proposal is accepted and so we may write for $x \in C$,
\begin{align*}
\P(x, \cdot) 
&=  \Q( \cdot \cap C^c ) + \Q(C) \P_C(x, \cdot)
\\
&= \Q( \cdot \cap C^c) + \frac{\Pi(C)}{u} \Pi_C(\cdot) + \left( \Q(C) - \frac{\Pi(C)}{u} \right) R_C(x, \cdot)
\\
&= \lambda_C \nu(\cdot) + (1 - \lambda_C) R_C(x, \cdot).
\end{align*}
So then
\[
\sup_{x, y \in C}\norm{\P(x, \cdot) - \P(y, \cdot)}_{TV}
\le 1 - \lambda_C.
\]
Let $\mu(\cdot) = \Q(\cdot \cap C^c) + \Q(C) \Pi_C(\cdot)$ and define the IMH process $(Y_t)_{t \ge 0}$ with $Y_0 \sim \mu$.
Let $S_T(\psi) = \sum_{ k = 0 }^{ T - 1 } I_{C}(X_k) \psi(X_k)$ and let $S'_T(\psi) = \sum_{ k = 0 }^{ T - 1 } I_{C}(Y_k) \psi(Y_k)$.
So then
\[
\norm{ \Q - \mu }_{TV}
\le \Q(C).
\]
We can control the bias with a coupling convergence argument.
On $C^c$, $X_0, Y_0$ are the same with probability $\Q(C^c)$ and on entering $C$, the split chain gives the rate to couple $1 - \lambda_C$.
The upper bound is
\[
\E \left| \frac{1}{T} S_T(\psi) - \frac{1}{T} S'_T(\psi) \right|
\le  \frac{ 2 \Q(C) }{T} \sum_{k = 0}^{T-1} ( 1 - \lambda_C )^k
\le \frac{2 \Q(C) }{ T \lambda_C }.
\]
We can also control the variance.
Define $\psi_C = \psi I_C$.
Then $\P^{t} \psi_C = I_C(x) \Q(C)^t \P_C^{t} \psi$ for $t \in \Z_+$.
For $i \le j$,
\[
\E( \psi_C(Y_i) \psi_C(Y_j) )
\le \left| \Q(C)^{j - i} \Pi_C( \psi_C \P_C^{j - i} \psi_C ) \right|
\le (1 - \lambda_C)^{j - i}.
\]
Then
\begin{align*}
\E \left| \frac{1}{T} S'_T(\psi) \right|^2
&\le \frac{1}{T} + \frac{2}{T} \sum_{k = 1}^{T-1}  \left( 1 - \frac{k}{T} \right)( 1 - \lambda_C )^k
\le \frac{ 2 }{ T \lambda_C }.
\end{align*}

The second term is controlled through an interesting reverse Lyapunov drift to $C$ on the density function.
For $k \in \Z_+$, let $D_k = \frac{d\Q \P^k}{d\Q}$ and $D_0 = 1$.
Then
\begin{align*}
\int_{B} D_{k + 1} d\Q
&= \Q \P^{k + 1}(B)
= \Q \P^{k} \P(B)
\\
&= \int_{B} \int_{\X} D_{k}(x) \alpha(x, y) \Q(dx) \Q(dy) + \int_{B} r(x) D_{k}(x) \Q(dx).
\end{align*}
Therefore for $\Q$ almost every $x \in \X$,
\[
D_{k + 1} (x)
= \int_{\X} D_{k}(y) \alpha(y, x) \Q(dy) +  r(x) D_{k}(x).
\]
For $x \not\in C$ and $y \in C$, $w(y) \le u < w(x)$ and $w(y) / w(x) \le 1$, and the acceptance is bounded below by
\[
\int_{\X} \alpha(x, y) \Q(dy)
\ge \int_{C} \frac{w(y)}{w(x)} \Q(dy)
= \frac{ \Pi(C) }{ w(x) }.
\]
By construction $\int_{\X} D_k(y) \Q(dy) = 1$.
We have by recursion for $x \not\in C$, $D_0 = 1$ and for $k \ge 1$,
\begin{align*}
D_{k} (x)
&\le 1 +  \left( 1 - \frac{ \Pi(C) }{ w(x) } \right) D_{k-1}(x)
\le \sum_{i = 0}^{k} \left( 1 - \frac{ \Pi(C) }{ w(x) } \right)^i.
\end{align*}
By Bernoulli's inequality,
\[
\sum_{i = 0}^{k} \left( 1 - \frac{ \Pi(C) }{ w(x) } \right)^i
\le \frac{ 1 - \left( 1 - \frac{ \Pi(C) }{ w(x) } \right)^{k + 1} }{ \frac{ \Pi(C) }{ w(x) }  }
\le (k + 1) \wedge \left( \frac{ w(x) }{ \Pi(C) } \right).
\]
Integrating over the initial condition for every $T \ge 1$,
\[
\E\left[ \frac{1}{T} \sum_{ k = 0 }^{ T - 1 } I_{C^c}(X_k) \right]
\le \frac{1}{T} \sum_{ k = 0 }^{ T - 1 } \int_{C^c} (k + 1) \wedge \left( \frac{ w(x) }{ \Pi(C) } \right) \Q(dx)
\le \int_{C^c} T \wedge \left( \frac{ w(x) }{ \Pi(C) } \right) \Q(dx).
\]

\end{proof}

This result allows us to prove a strong result for sequences of probability measures, but the result is less descriptive.

\begin{lemma}
\label{lemma:general_cutoff_profile}
Let $T_n \to \infty$ and assume there is an $L_n \to \infty$ such that
\begin{align}
&\lim_{n \to \infty} \Pi_n\left( \frac{w_n}{T_n}\le \frac{1}{L_n} \right) = h \in (0, 1]
&\lim_{n \to \infty} \Pi_n\left( \frac{1}{L_n} \le \frac{w_n}{T_n}\le L_n \right) = 0.
\label{eq:abstract_conditions}
\end{align}
Then
\[
\lim_{n \to \infty} R_{T_n}(\Q_n, \Pi_n) = 2 (1 - h).
\]
\end{lemma}

\begin{proof}
Assume $n \in \Z_+$ is large enough so $T_n L_n > 1$.
Define the set $C_n = \{ w_n / T_n \le L_n \}$.
For the lower bound, using the lower bound in Theorem~\ref{theorem:upper_lower_bound} with Markov's inequality and Bernoulli's inequality,
\begin{align*}
R_{T_n}(\Q_n, \Pi_n)
&\ge 2 \left[ 1 - \Pi_n\left( C_n \right) \right] \left( 1 - \Q_n\left( C_n^c \right) \right)^{T_n}
\\
&\ge 2 \left[ 1 - \Pi_n\left( C_n \right) \right] \left( 1 - \frac{1}{L_n T_n} \right)^{T_n}
\\
&\ge 2 \left[ 1 - \Pi_n\left( C_n \right) \right] \left( 1 - 1/L_n \right).
\end{align*}
Using the assumed conditions \eqref{eq:abstract_conditions}, as $n \to \infty$,
\[
\Pi_n\left( C_n^c \right)
= 1 - \Pi_n\left( \frac{w_n}{T_n} \le 1/L_n \right) - \Pi_n\left( 1/L_n < \frac{w_n}{T_n} \le L_n \right)
\to 1 - h.
\]
Therefore,
\[
\liminf_{n \to \infty} R_{T_n}(\Q_n, \Pi_n) \ge 2(1 - h).
\]

Define the set $D_n = \{ w_n / T_n \le 1/L_n \}$.
Let $\lambda_{n} = 1 - \Q_n(D_n) + L_n \Pi_n(D_n) / T_n$.
For the upper bound, we use Theorem~\ref{theorem:upper_lower_bound},
\begin{align*}
R_{T_n}(\Q_n, \Pi_n)
\le 
\sqrt{ \frac{2}{ T_n \lambda_{n} } } + \frac{2 \Q_n(D_n) }{ T_n \lambda_{n} }
+ 2 \int_{D_n^c} T_n \wedge \frac{w}{\Pi(D_n)} d\Q_n
+ 2 \Pi(D_n^c).
\end{align*}
Then $\lambda_{n} T_n \to \infty$ and the first two terms tend to $0$ as $n \to \infty$.
By \eqref{eq:abstract_conditions}, the last term $2 \Pi(D_n^c) \to 2 (1 - h)$.
It remains to show the remaining term vanishes and we have
\[
\int_{ \left\{ \frac{w_n}{T_n} > \frac{1}{L_n} \right\} } T_n \wedge \frac{w_n}{\Pi_n(D_n)} d\Q_n
= \int_{ \left\{ \frac{w_n}{T_n} > \frac{1}{L_n} \right\} } \frac{T_n}{w_n} \wedge \frac{1}{\Pi_n(D_n)} d\Pi_n
\]
and
\begin{align*}
\int_{ \left\{ \frac{w_n}{T_n} > \frac{1}{L_n} \right\} } \frac{T_n}{w_n(x)} \wedge \frac{1}{\Pi_n(D_n)} \Pi_n(dx)
&\le \int_{ \left\{ \frac{1}{L_n} < \frac{w_n}{T_n} \le L_n \right\} } \frac{T_n}{w_n(x)} \wedge \frac{1}{\Pi_n(D_n)} \Pi_n(dx)
\\
&\quad+ \int_{ \left\{ L_n < \frac{w_n}{T_n} \right\} } \frac{T_n}{w_n(x)} \wedge \frac{1}{\Pi_n(D_n)} \Pi_n(dx)
\\
&\le \frac{1}{\Pi_n(D_n)} \Pi_n\left( \frac{1}{L_n} < \frac{w_n}{T_n} \le L_n  \right) 
+ \frac{1}{L_n}.
\end{align*}
By the assumed conditions \eqref{eq:abstract_conditions}, both terms tend to $0$ as $n \to \infty$.
Therefore,
\[
\limsup_{n \to \infty} R_{T_n}(\Q_n, \Pi_n) \le 2(1 - h).
\]
\end{proof}

We are now in a position to prove Theorem~\ref{theorem:convergence_characterization} and Theorem~\ref{theorem:cutoff_result}.

\begin{proof}[Proof of Theorem~\ref{theorem:convergence_characterization}]
Assume $w_n / T_n \not\to 0$ in probability.
Then there are $\e, \delta \in (0, 1)$ and a subsequence satisfying $\Pi_{n_k}( w_{n_k} / T_{n_k} \ge \e) \ge \delta$.
Using the lower bound in Theorem~\ref{theorem:upper_lower_bound} and Markov's inequality, then for all sufficiently large $n$,
\begin{align*}
R_{T_{n_k}}(\Q_{n_k}, \Pi_{n_k})
&\ge 2 \Pi_{n_k}( w_{n_k} / T_{n_k} \ge \e) \left( 1 - \frac{1}{\e T_{n_k}} \right)^{ T_{n_k} }
\\
&\ge 2 \delta \left( 1 - \frac{1}{\e T_{n_k}} \right)^{ T_{n_k} }
\\
&\to 2 \delta \exp(-1/\e).
\end{align*}
Therefore, $R_{T_{n}}(\Q_{n}, \Pi_{n}) \not\to 0$. 

Now assume $w_n / T_n \to 0$ in probability.
Let $C_n = \{ w_n \le T_n \e \}$.
Then for any $\e > 0$, $\Pi_n(C_n^c) \to 0$.
Let $\lambda_{C_n} = 1 - \Q_n(C_n) + \Pi_n(C_n) / (\e T_n)$.
The upper bound in Theorem~\ref{theorem:upper_lower_bound} is
\begin{align*}
R_{T_n}(\Q_n, \Pi_n)
&\le \sqrt{ \frac{2}{ T_n \lambda_{C_n} } } + \frac{2 \Q(C_n) }{ T_n \lambda_{C_n} }
+ 2 \int_{C_n^c} T_n \wedge \frac{w_n}{\Pi(C_n)} d\Q_n
+ 2 \Pi(C_n^c)
\\
&\le \sqrt{ \frac{2 \e}{ \Pi(C_n) } } + \frac{2 \e \Q(C_n) }{ \Pi(C_n) }
+ 2 \frac{\Pi(C_n^c) }{ \Pi(C_n) } 
+ 2 \Pi(C_n^c).
\end{align*}
Since $\Pi(C_n) \to 1$ and $\e > 0$ is arbitrary, then $R_{T_{n}} \to 0$. 
\end{proof}

\begin{proof}[Proof of Theorem~\ref{theorem:cutoff_result}]
We first prove the logarithmic cutoff \eqref{eq:cutoff}.
Let $\e, \delta > 0$ and define 
\begin{align*}
&M^{+}_n(\delta) = \lfloor \exp[ (1 + \delta) \H(\Pi_n, \Q_n) ] \rfloor,
&M^{-}_n(\delta) = \lfloor \exp[ (1 - \delta) \H(\Pi_n, \Q_n) ] \rfloor.
\end{align*}
Let $0 < s < \delta$ and choose $L_n = \exp(s \H(\Pi_n, \Q_n))$.
Then
\[
\Pi_n(\log(w_n) - \log(M^{+}_n(\delta)) \le - \log(L_n))
\to 1.
\]
Therefore Lemma~\ref{lemma:general_cutoff_profile} yields $T_\e(\Q_n, \Pi_n) \le M^{+}_n(\delta)$ for sufficiently large $n$.
Now for the lower bound.
Let $B_n = \{ w_n \ge \exp[ (1-\delta/2) \H(\Pi_n, \Q_n) ] \}$ so that by the consistency assumption, $\Pi_n(B_n) \to 1$.
With Markov's inequality and Theorem~\ref{theorem:upper_lower_bound}, for any $T \le M^{-}_n(\delta)$
\begin{align*}
R_{T}(\Q_n, \Pi_n)
&\ge 2 \Pi_n(B_n) (1 - \Q_n(B_n))^T
\\
&\ge 2 \Pi_n(B_n) (1 - T \Q_n(B_n))
\\
&\ge 2 \Pi_n(B_n) \left\{ 1 - T \exp\left[ -(1-\delta/2) \H(\Pi_n, \Q_n) \right] \right\}
\\
&\to 2.
\end{align*}
It follows that for sufficiently large $n$
\[
M^{-}_n(\delta)\le
T_\e(\Q_n, \Pi_n)
\le M^{+}_n(\delta).
\]
Therefore taking $\delta \to 0$, then we obtain the conclusion.

We now prove the cutoff profile \eqref{eq:cutoff_profile}.
The Gaussian limit assumption implies 
\[
\E_{\Pi_n}(\log(w_n)) / \Var_{\Pi_n}( \log(w_n) ]^{1/2} \to \infty
\] 
or the Gaussian limit would fail.
For $z \in \R$, this implies $T_n(z) \to \infty$. 
Choose $L_n = \exp( [ \Var_{\Pi_n}( \log(w_n) ] )^{1/4} )$ and by assumption $L_n \to \infty$.
We will show the conditions of Lemma~\ref{lemma:general_cutoff_profile} hold.
By the assumed condition,
\begin{align*}
&\Pi_n\left( w_n / T_n(z) \le 1/L_n \right)
\\
&= \Pi_n\left( 
\frac{ \log(w_n) - \E_{\Pi_n}(\log(w_n)) }
{ \sqrt{ \Var( \log(w_n) ) } } 
\le \frac{ \log( T_n(z) / L_n ) - \E_{\Pi_n}(\log(w_n)) }
{ \sqrt{ \Var( \log(w_n) ) } } 
\right).
\end{align*}
Now,
\begin{align*}
\frac{ \log( T_n(z) / L_n ) - \E_{\Pi_n}(\log(w_n)) }{ \sqrt{ \Var( \log(w_n) ) } }
\to z.
\end{align*}
Since $\Phi$ is continuous, then 
$
\Pi_n\left( w_n/T_n(z) \le 1/L_n \right)
\to \Phi(z).
$
Similarly, 
\begin{align*}
\frac{ \log( T_n(z) L_n ) - \E_{\Pi_n}(\log(w_n)) }{ \sqrt{ \Var( \log(w_n) ) } }
\to z
\end{align*}
and so
$
\Pi_n\left( w_n/T_n(z) \le L_n \right)
\to \Phi(z).
$
Therefore, $\Pi_n\left( 1/L_n \le w_n/T_n(z) \le L_n \right) \to 0$.
\end{proof}

\section{Applications}
\label{section:applications}

We consider some informative examples that highlight the theory developed.
The Gaussian mismatch examples illustrate proposal mismatches to the target which are often exhibited in practice.
The Bayesian concentration example highlights the computability of the logarithmic cutoff profile.

\subsection{Application: Gaussian mean mismatches}

For $d \in \Z_+$, let $\X_d = \R^d$ and consider $\Pi_d \equiv N(\mu_d, I_d)$ and $\Q_d \equiv N(0, I_d)$ where $\norm{\mu_d} \to \infty$.
A standard computation yields
\[
\log(w_d(x))
= \mu_d^T x - \frac{1}{2} \norm{\mu_d}^2.
\]
Then $\H(\Pi_d, \Q_d) = \frac{1}{2} \norm{\mu_d}^2$ and $\Var_{\Pi_d}[ \log(w_d(x)) ] = \norm{\mu_d}^2$, so
\[
\frac{ \log(w_d(x)) - \frac{1}{2} \norm{\mu_d}^2 }
{ \norm{\mu_d} }
\]
is a standard $N(0, 1)$ random variable.
Theorem~\ref{theorem:cutoff_result} yields the cutoff profile
\[
\lim_{d \to \infty} R_{T_d(z)}(\Q_d, \Pi_d) = 2 (1 - \Phi(z))
\]
where for $z \in \R$, 
\[
T_d(z) = 1 \vee \left\lfloor \exp\left( \frac{1}{2} \norm{\mu_d}^2 + z \norm{\mu_d} \right) \right\rfloor.
\] 
In particular, choose $\mu_d \equiv (1, \ldots, 1)^T \in \R^d$ and the number of samples required can be exponential in the dimension since
$
T_d(z) = 1 \vee \left\lfloor \exp\left( \frac{d}{2} + z \sqrt{d} \right) \right\rfloor.
$

\subsection{Application: Gaussian variance mismatches}

Let $\X = \R^d$ and consider $\Pi_d \equiv N(0, I_d)$ and $\Q_d \equiv N(0, \tau^2 I_d)$ where $\tau > 0$ and $\tau \not= 1$.
We have
\begin{align*}
\log(w_d(x)) = d \log(\tau) + \frac{1/\tau^2 - 1}{2} \norm{x}^2.
\end{align*}
Under $\Pi_d$, $\norm{X}^2$ has mean $d$ and variance $2d$.
Then
\[
\frac{\log(w_d) - \E_{\Pi_d}[ \log(w_d) ]  }{\sqrt{ \Var_{\Pi_d}[ \log(w_d) ] }}
= \operatorname{sgn}(1/\tau^2 - 1)
\left[ \frac{\norm{X}^2- d}{\sqrt{2d}} \right]
\to N(0,1)
\]
weakly as $d \to \infty$.
For $z \in \R$, let 
\[
T_d(z) = 1 \vee \left\lfloor \exp\left( d \log(\tau) + \frac{d( 1/\tau^2 - 1 )}{2} + z \sqrt{ \frac{d}{2} } | 1/\tau^2 - 1 |  \right) \right\rfloor
\]
and with $\Phi$ the distribution function of a Gaussian, Theorem~\ref{theorem:cutoff_result} yields the cutoff profile
\[
\lim_{d \to \infty} R_{T_d(z)}(\Q_d, \Pi_d) = 2 (1 - \Phi(z)).
\]
In particular, the number of samples required is exponential in the dimension.

\subsection{Application: Bayesian Posterior Concentration}

Consider Bayesian logistic regression with a Gaussian prior and the IMH algorithm with an independent Laplace proposal.
Let $N \in \Z_+$ denote the size of the data and let $d \equiv d_N \in \Z_+$ denote the dimension of the features.
We will use the logarithmic cutoff profile to determine how many IMH iterations are needed to correct the approximation when computing posterior integrals.
Let $\tau \in (0, \infty)$ be a fixed variance parameter and we will assume a Gaussian prior $\beta \sim N(0, \tau^2 I_d)$.
Let $Y_i \in \{0, 1\}$ and $x_i \in \R^d$ for $i \in 1, \ldots, N$ and the posterior $\Pi_N \propto \exp(- U_N(\beta) ) d\beta$ has potential
\begin{align}
U_N(\beta)
&= \sum_{ i = 1 }^N \left[ \log( 1 + \exp( x_i^T \beta ) ) - Y_i x_i^T\beta \right] + \frac{1}{2 \tau^2} \norm{\beta}^2.
\end{align}
Let $\widehat{\beta}_N$ denote its unique minimal point, which exists.
Let $\widehat{H}_N=\nabla^2 U_N(\widehat{\beta}_N)$ and define the Laplace proposal
$
\Q_N \equiv N(\widehat{\beta}_N, \widehat{H}_N^{-1}).
$

We will consider the following assumptions on the data used in the posterior.
Assume the design rows $x_1,\ldots,x_N$ are independent
and have the form
\[
x_i = ( 1, Z^i_1, \ldots, Z^i_{d-1} )^T,
\qquad
\E(Z^i_j) = 0, \quad \E({Z^i_j}^2) = 1, \quad |Z^i_j| \le K
\]
for some fixed $K \in (0, \infty)$ and for each $i$, the coordinates $Z^i_1,\ldots,Z^i_{d-1}$ are independent.
Let $\sigma(t)=1/(1+e^{-t})$.
Assume for some fixed $B \in (0, \infty)$ and $\beta^*_{d}$ with $\norm{ \beta^*_{d} } \le B$, the responses are conditionally independent with
\begin{align}
Y_i \mid (x_j)_{j=1}^N \sim \text{Bernoulli}\bigl( \sigma( x_i^T \beta^*_{d} ) \bigr).
\end{align}
Write $\Prob_*$ and $\E_*$ for sampling probability and expectation
for $(x_i,Y_i)_{i=1}^N$.
The following constructs a cutoff profile that can be computed from the data in a simulation where the risk \eqref{eq:risk} studied here is conditioned on the data.

\begin{theorem}
\label{theorem:laplace_cutoff}
Let $(x_i, Y_i)_{i = 1}^N$ be samples from $\Prob_*$. Conditioned on the data, let $\Pi_N$ denote the Bayesian logistic posterior and $\Q_N$ the Laplace proposal defined in this section.
Assume $d^{5/3}/N\to0$.
Define the fitted shift vector
\begin{align}
v_N
&=\frac12\widehat H_N^{-1/2}
\sum_{i=1}^N\sigma''(x_i^T\widehat\beta_N)
(x_i^T\widehat H_N^{-1}x_i)x_i
\end{align}
and assume
$
\norm{v_N}\to\infty
$
in $\Prob_*$ probability.
For every $z \in \R$, with
$
T_N(z) = 1 \vee \left\lfloor \exp\left( \norm{v_N}^2/2 + z \norm{v_N} \right) \right\rfloor,
$
\begin{align*}
R_{T_N(z)}( \Q_N, \Pi_N )
&\to 2 ( 1 - \Phi(z) )
\end{align*}
in $\Prob_*$ probability.
\end{theorem}

\section{Absence of Standard Cutoff}
\label{section:failure_of_total_variation_cutoff}

We motivate the logarithmic cutoff theory in the $L^1$ risk \eqref{eq:risk} and prove that there is no total variation cutoff for the marginal distribution of the IMH process.
Define the total variation for two Borel probability measures $\nu, \nu'$ on $\X$
\[
\norm{ \nu - \nu' }_{TV}
= \sup_{f : \X \to [0, 1]} \left| \int f d\nu - \int f d\nu' \right|
\]
where the functions in the supremum are Borel measurable.
For $n \in \Z_+$ and $\e \in (0, 1)$, define
\begin{align*}
\tau_n(\e) = \inf\{ t \in \Z_+ \cup \{ 0 \} : \norm{ \Q_n \P_n^t - \Pi_n }_{TV} \le \e \}
\end{align*}
where $\inf \varnothing = \infty$.
Convergence of the IMH process exhibits a sharp transition to convergence or cutoff in total variation on a standard time scale if for $\e \in (0, 1/2)$, $\tau_n(1-\e) \to \infty$ and
\[ 
\lim_{n \to \infty} \frac{ \tau_n(\e) }{ \tau_n(1-\e) } = 1.
\]
The following shows this is not possible. Figure~\ref{fig:no_cutoff} provides a visualization.

\begin{figure}[t]
\centering
\includegraphics[width=0.8\textwidth]{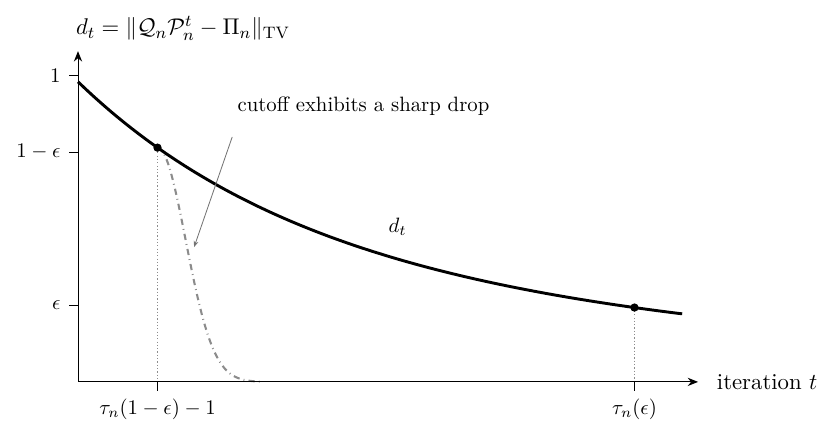}
\caption{A visualization of the main proof idea for Theorem~\ref{theorem:no_cutoff} showing the decay of total variation from time $\tau_n(1 - \e)$ to $\tau_n(\e)$ must be smooth.}
\label{fig:no_cutoff}
\end{figure}

\begin{theorem}
\label{theorem:no_cutoff}
For any $\e \in (0, 1/2)$, if $\tau_n(1-\e) \to \infty$, then
\[
\liminf_{n \to \infty} \frac{ \tau_n(\e) }{ \tau_n(1-\e) } 
> 1.
\]
\end{theorem}

\begin{proof}
We decompose $\Q\P^k$ into its singular part with respect to $\Pi$ and the residual part using ideas from \citep{smith:tierney} to get a multiplicative decomposition like a spectral decomposition.
For $t \in (0, \infty]$, define
\[
b(t) = \int_{\X} (1 - w/t)_+ \Q
\]
and $b(0) = \lim_{t \downarrow 0} b(t)$.
For $r > 0$ and integers $k \ge 0$, define $h_k(r) = \int_{r}^{\infty} \frac{ (k + 1) b(u)^k }{ u^2 } du$.
Using a change of variables \citep[Theorem 1 and Lemma 3]{smith:tierney},
we can write for $k \ge 0$
\[
\Q \P^k(dy) = b(0)^k I_{\{ w = 0 \} } \Q(dy) + h_k( w(y) ) I_{\{ w > 0 \} } \Pi(dy).
\]
We defined this so that at $k = 0$, this is well defined.

Define the pushforward measures $\eta = \Q \circ w^{-1}$ and $\mu = \Pi \circ w^{-1}$ where $w^{-1}$ is the inverse image of the weight function $w$.
For integers $k \ge 0$, define $\eta_k = \Q \P^k \circ w^{-1}$.
Using a change of variables like \citep[Theorem 1 and Lemma 3]{smith:tierney}, we can get
\[
\eta_k(dr) = b(0)^{k + 1} \delta_0(dr) + h_{k}(r) \mu(dr).
\]
We will use this decomposition to show that $\tau_n(\e) , \tau_n(1-\e) < \infty$ for every fixed $n \in \Z_+$.
Define $d_{k} = \norm{\Q \P^k - \Pi}_{TV}$.
Define $D_k(t) = \eta_k([0,t]) - \mu([0,t])$ for $t \in [0, \infty)$.
Since $h_k$ is nonincreasing, taking the supremum, we get the identity
\[
d_{k} = \sup_{t \ge 0} D_k(t)
= \int_{(0, \infty)} (1 - h_k(r))_+ \mu(dr).
\]
For $t \in [0, \infty)$, define $M(t) = \mu([0, t])$ and $M(u-) = \lim_{v \uparrow u} M(v)$.
Then also by \citep[Equation 7]{smith:tierney}, if $M(r) > 0$,
\[
h_{k}(r) = 1 + \int_{(r, \infty)} \frac{ b(u)^{k + 1} }{ M(u) M(u-) } \mu(du) - \frac{ b(r)^{k + 1} }{ M(r) }.
\]
So then,
\[
(1 - h_{k}(r))_+
\le 1  \wedge \frac{ b(r)^{k + 1} }{ M(r) }
\to 0.
\]
By dominated convergence, then $d_k \to 0$ and $\tau_n(\e) , \tau_n(1-\e) < \infty$ for every fixed $n \in \Z_+$.

We construct a type of multiplicative decomposition analogous to a spectral decomposition.
If $M(t) = 0$, then $D_k(t) = b(0)^{k + 1}$.
If $M(t) > 0$, then using Fubini's theorem
\begin{align*}
D_k(t) 
&= \mu((t,\infty)) - \eta_k((t, \infty))
\\
&= \int_{ r \in (t, \infty) } \frac{ b(r)^{k + 1} }{ M(r) } \mu(dr) - \int_{ r \in (t, \infty) }\int_{ u \in (r, \infty) } \frac{ b(u)^{k + 1} }{ M(u) M(u-) } \mu(du) \mu(dr)
\\
&= \int_{ r \in (t, \infty) } \frac{ b(r)^{k + 1}  }{ M(r) } \mu(dr) 
- \int_{ u \in (t, \infty) } \frac{ b(u)^{k + 1}  [M(u-) - M(t)]}{ M(u) M(u-) } \mu(du)
\\
&=  M(t) \int_{ r \in (t, \infty) } \frac{ b(r)^{k + 1} }{ M(r)  M(r-) } \mu(dr).
\end{align*} 
We have proven that there is a type of multiplicative decomposition where there is a subprobability measure $\nu_t$ such that
\[
D_k(t) = \int_{[0, 1)} v^k \nu_t(dv)
\]
and $\nu_t([0, 1)) = D_0(t)$.
In particular, for every Borel $A \subseteq [0,1)$, the measure is defined by
\[
\nu_t(A)=
\begin{cases}
 M(t)\displaystyle\int_{(t,\infty)}
       \mathbf 1_{\{b(u)\in A\}}
       \frac{b(u)}{M(u)M(u-)}\,\mu(du),&M(t)>0,\\[2ex]
 b(0) \,\mathbf 1_A(b(0)), &M(t)=0.
 \end{cases}
\]
We use Jensen's inequality as in \citep[Lemma 2.2]{rockner:wang:2001}.
If $D_0(t) = 0$, then $D_k(t) = 0$ for $k \ge 0$ and so we will assume $D_0(t) > 0$.
This gives for $1 \le M \le T$,
\begin{align}
D_M(t) / D_0(t) &= \int_{[0, 1)} (v^T)^{M/T} \nu_t(dv) / D_0(t)
\le \left( \int_{[0, 1)} v^T \nu_t(dv) / D_0(t) \right)^{M/T}
\\
D_M(t) &\le D_T(t)^{M/T} D_0(t)^{1 - M/T} \le D_T(t)^{M/T}
\\
d_{T} &\ge d_{M}^{T/M}.
\label{eq:jensen_tv}
\end{align}
Now define for $n, k \in \Z_+$, $d_{n, k} = \norm{\Q_n \P_n^k - \Pi_n}_{TV}$.
Applying \eqref{eq:jensen_tv} and using the definition of the mixing times and the fact that the mixing times are finite, we have
\begin{align*}
(1 - \e)^{ \frac{ \tau_n(\e) }{ \tau_n(1-\e) - 1 }} 
< d_{n, \tau_n(1-\e) -1}^{ \frac{ \tau_n(\e) }{ \tau_n(1-\e) - 1 }} 
\le d_{n, \tau_n(\e)} \le \e.
\end{align*}
Taking the logarithms gives
\[
\frac{ \tau_n(\e) }{ \tau_n(1-\e) } \ge \left( 1 - \frac{ 1 }{ \tau_n(1-\e) } \right) \frac{ \log(\e) }{ \log(1 - \e) }.
\]
Since $\lim_{n \to \infty} \tau_n(1-\e) = \infty$, taking the limit completes the result.
\end{proof}

\section{Unreliability of Convergence Diagnostics}
\label{section:diagnostic_failure}

The goal of this section is to further motivate the previously developed high-dimensional convergence theory.
Convergence diagnostics are often used to validate convergence for Metropolis-Hastings independence samplers.
The hope is that this could allow bypassing a convergence analysis entirely.
We will show that a convergence diagnostic has a nontrivial probability to validate both an accurate simulation and an inaccurate simulation.
In other words, we show convergence diagnostics based on the simulation output do not provide a reliable diagnostic in general. 

An interesting example is the formulation of Gelman's $\hat{R}$ \citep{gelman:rubin:1992} for testing convergence in MCMC.
For an index $j = 1, \ldots, J$, consider multiple IMH processes $(X^j_n(t))_{t \ge 0}$ each with Markov kernel $\P_n$ with target $\Pi_n$ and proposal $\Q_n$. 
We use the variance-ratio statistic \citep{vehtari:etal:2021}. 
For $J \ge 2$ independent chains of length $k \ge 2$ states, let $f : \X \to \R$ be Borel measurable and write $f_{j,i} = f(X^j(i))$, and define the sample mean and sample variance 
$
\bar{f}_j =\frac1k\sum_{i=0}^{k-1} f_{j,i}
$
and
$S_j^2(f) = \frac1{k-1} \sum_{i=0}^{k-1}(f_{j,i}-\bar{f}_j)^2.
$
Define the mean of the sample means
$
\bar{f} = \frac{1}{J} \sum_{j=1}^J \bar f_j
$
and the scaled sample variance of the sample means
$
\overline{B}(f) =\frac{k}{J-1}\sum_{j=1}^J (\bar{f}_j - \bar{f})^2
$
and the mean of the sample variances
$
\overline{S^2}(f) =\frac1J \sum_{j=1}^J S^2_j(f).
$
The $\hat{R}(f)$ convergence diagnostic is defined by
\begin{align*}
\hat{R}(f) =\sqrt{\frac{(k-1) \overline{S^2}(f) + \overline{B}(f)}{k \overline{S^2}(f)}}
\end{align*}
if $\overline{S^2}(f) > 0$ and $\hat{R}(f) = \infty$ if $\overline{S^2}(f) = 0$. 
A value $\hat{R}(f) \approx 1$ would mean the diagnostic validates convergence.

Consider an alternative sequence of IMH kernels $\K_n$ with target $\nu_n$ and proposal $\Q_n$.
Let $B_n \subseteq \X_n$ be a sequence of Borel sets such that the weights differ only on $B_n$, that is, $d\Pi_n / d\Q_n \propto d\nu_n / d\Q_n$ on $B_n^c$ and such that convergence happens
\begin{align}
&R_{k_n}(\Q_n, \nu_n) \to 0
\label{eq:comparison_assumption}
\end{align}
as $n \to \infty$.

\begin{theorem}
\label{theorem:Rhat_failure}
Assume $B_n$ is a sequence of Borel sets and $\K_n$ is an alternative IMH kernel with target $\nu_n$ satisfying \eqref{eq:comparison_assumption}.
Fix $J \ge 2$ and let $k_n \ge 2$ such that $k_n \to \infty$.
Assume the proposal visits the sets $B_n$ so that
\begin{align*}
&\rho_n = 1 - \{ 1 - \Q_n(B_n) \}^{J k_n} \to 0
\end{align*} 
as $n \to \infty$.
Let $f_n : \X_n \to \R$ be a sequence of Borel measurable bounded functions with $\norm{ f_n }_{\infty} \le 1$.
Assume for some $\sigma^2_* > 0$, the variance is not degenerate,
$
\Var_{\nu_{n}}( f_n ) \ge \sigma^2_*
$
and the means differ, that is, for some $\Delta > 0$ and all sufficiently large $n \in \Z_+$,
\begin{align*}
&|\Pi_n(f_n) - \nu_n(f_n)| \ge \Delta.
\end{align*}
Then for every fixed $\tau > 1$ and $0 < \delta < \Delta$, the $\hat{R}(f_n)$ diagnostic gives a false positive with high probability for the original IMH process
\begin{align*}
\Prob\left(
\hat{R}(f_n) \le \tau, \quad
\min_{1\le j\le J} \left| \frac{1}{k_n} \sum_{i = 0}^{k_n - 1} f_{n}(X_{j}(i)) - \Pi_n(f_n) \right| > \delta \right) \to 1
\end{align*}
as $n \to \infty$.
\end{theorem}

We first prove a more general result.
Let $\Q$ be a proposal measure and $\Pi$ the target measure and $\nu$ a comparison measure.
Define the weights $w = d\Pi / d\Q$ and $v = d\nu / d\Q$.
Run $J \in \Z_+ $ independent chains, each initialized from $\Q$, with fixed transition counts $t_1, \ldots, t_J \in\mathbb \Z_+$.
Let $\e > 0$ and let $\hat{D} \in \{0, 1\}$ be a diagnostic from the entire samples of $J$ Markov chains of length $t_{j} + 1$.
The number of independent raw draws from $\Q$, counting initializations, is
\[
M = \sum_{ j = 1 }^J ( t_j + 1 ).
\]
For $1 \le k \le t_j + 1$, set
\[
A_{j,k}(f) = 
\frac{1}{k} \sum_{ i = 0 }^{ k - 1 } f( X_{j, i} ).
\]
For a specified integrable quantity $f : \X \to \R$ and tolerance $\delta > 0$,
\[
\{ \hat{D} = 1 , | A_{1,k}(f) - \Pi(f) | > \delta \}
\]
is the set of false convergence validation.
Let $\P$ be the IMH Markov kernel with target $\Pi$ and $\K$ with target $\nu$.
Let $\Prob_\P$ denote the probability with respect to $(X_t)_{t = 0}^{\infty}$ generated from the kernel $\P$ and similarly define $\Prob_\K$ using the kernel $\K$.

\begin{theorem}
\label{theorem:general_failure}
Let $v$ be a comparison weight and $w$ the weight of interest, with
$w \propto v$ on $B^c$ for some Borel set $B \subseteq \X$ with $\Pi(B) > 0$.
Define
$
\rho = 1 - \{ 1 - \mathcal{Q}(B) \}^{M}.
$
Let $f : \X \to \R$ be Borel measurable with $\int |f| d\Pi, \int |f| d\nu < \infty$.
Let
$
\Delta = | \Pi(f) - \nu(f) |.
$
For every $0 < \delta < \Delta$,
\begin{align*}
&\Prob_{\P}\left( \hat{D} = 1, \min_{j} | A_{j, t_j + 1}(f)
- \Pi(f) | \ge \Delta - \delta \right)
\\
&\ge \Prob_{\K}( \hat{D} = 1 ) - \Prob_{\K}( \max_{j} | A_{j,  t_j + 1}(f) -  \nu(f) | > \delta ) - \rho.
\end{align*}
\end{theorem}

\begin{proof}[Proof of Theorem~\ref{theorem:general_failure}]
For each $j$, let $Y_{j, 0},\ldots,Y_{j, t_j}$ be independent samples from
$\mathcal Q$ and let $U_{j, k}$ be independent uniform random variables. 
Initialize at $Y_{j,0}$, and use the same proposal $Y_{j,i}$ and uniform for update $i$.
Denote the coupling of $(X_{j, k})$ with target $\Pi$ and $(\tilde{X}_{j, k})$ with target $\nu$ by $\Prob_{2}$.
For $1 \le k \le t_j + 1$, define
\begin{align*}
&A_{j,k}(f) = 
\frac{1}{k} \sum_{ i = 0 }^{ k - 1 } f( X_{j, i} )
&\tilde{A}_{j,k}(f) = 
\frac{1}{k} \sum_{ i = 0 }^{ k - 1 } f( \tilde{X}_{j, i} ).
\end{align*}
Let $\tilde{D}$ denote the convergence diagnostic computed only from the process $(\tilde{X}_{j, k})$ and define the sets
\begin{align*}
&E = \{ \tilde{D} = 1 \}
&F = \left\{ \max_{j} | \tilde{A}_{j, t_j + 1}(f) - \nu(f) | \le \delta \right\}.
\end{align*}
Define the set
$
H = \{ Y_{j,i} \notin B \text{ for all } j \text{ and } 0 \le i \le t_j \}.
$
By construction
$
\Prob_{2}(H) = \{ 1-\mathcal \Q(B) \}^{M} = 1 - \rho.
$

On $H$, both chains initialize the same and are the same for all $j$,
\begin{align*}
&\hat{D} = \tilde{D}
&A_{j, t_j + 1}(f) = \tilde{A}_{j, t_j + 1}(f).
\end{align*}

On $E \cap F \cap H$, for all $j$,
\[
|A_{j, t_j + 1}(f) - \Pi(f)|
= |\tilde{A}_{j, t_j + 1}(f) - \Pi(f)|
 \ge | \nu(f) - \Pi(f) | - | \tilde{A}_{j, t_j + 1}(f) - \nu(f) |
 \ge \Delta - \delta.
\]
We have the lower bound
\begin{align*}
&\Prob_{\P}\left( \hat{D} = 1,\ \min_{j} |A_{j, t_j + 1}(f) - \Pi(f) |\ge \Delta - \delta \right)
\ge\Prob_{2}( E \cap F \cap H)
\\
&\ge \Prob_{2}(E) - \Prob_{2}\left( E \cap(F^c \cup H^c) \right)
\ge \Prob_{2}(E) - \Prob_{\K}(F^c)- \Prob_{2}(H^c)
\\
&= \Prob_{\K}( \tilde{D} = 1 ) -  \Prob_{\K}\left( \max_{j} | \tilde{A}_{j, t_j + 1}(f) - \nu(f) |  > \delta \right) - \rho.
\end{align*}
\end{proof}

\begin{proof}[Proof of Theorem~\ref{theorem:Rhat_failure}]
Let $G_n$ be the event that both errors are at most $\eta > 0$ in every chain:
\[
G_n = \bigcap_{ j = 1 }^J
\left\{ | A_{j, k_n}(f_n) - \nu_n(f_n) | \le \eta, | A_{j, k_n}(f_n^2) - \nu_n(f_n^2) | \le \eta \right\}.
\]
Markov's inequality gives
\[
\Prob_{\K_n}( G_n^c )
\le \frac{2 J R_{k_n}(\Q_n, \nu_n)}{\eta}
\to 0.
\]

On $G_n$, since $|A_{j, k_n}(f_n)|, | \nu_{n}(f_n) | \le 1$,
\[
 \left| A_{j, k_n}(f_n)^2 - \nu_{n}(f_n)^2 \right|
 \le 2 | A_{j, k_n}(f_n) - \nu_{n}(f_n) | 
 \le 2 \eta.
\]
On $G_n$, then we can choose $\eta > 0$ small enough depending on $\sigma^2_*$ so that
\[
A_{j, k_n}(f_n^2) - ( A_{j, k_n}(f_n))^2
\ge \nu_n(f_n^2) - \nu_{n}(f_n)^2 - 3\eta
= \Var_{\nu_n}(f_n) - 3 \eta
\ge \sigma_*^2 / 2.
\]
On $G_n$,
$
\overline{B}_n(f_n) / k_n
\le \frac{J}{J - 1} \eta^2.
$
We can choose $\eta$ in terms of $\tau$ so that we have the upper bound
\begin{align*}
\hat{R}(f_n)^2
= 1 - \frac{1}{k_n} + \frac{\overline{B}_n(f_n)}{k_n} \frac{1}{\overline{S^2}_n(f_n)}
\le 1 - \frac{1}{k_n} + \frac{\overline{B}_n(f_n)}{k_n} \frac{2}{\sigma^2_*}
\le \tau^2.
\end{align*}
For sufficiently large $n \in \Z_+$, then
\[
\Prob_{\K_n}\left( \hat{R}(f_n) \le \tau \right) 
\ge \Prob_{\K_n}( G_n )
\ge 1 - \e.
\]
For any $\e \in (0, 1)$ and $0 < \delta < \Delta$ and $j \in \{1, \ldots, J \}$, for $n \in \Z_+$ sufficiently large, Theorem~\ref{theorem:general_failure} gives
\begin{align*}
\Prob_{\P_n}\left( \hat{R}(f_n) \le \tau, | A_{j, k_n}(f_n)
- \Pi_n(f_n) | \ge \Delta - \delta \right)
&\ge \Prob_{\K_n}\left( \hat{R}(f_n) \le \tau \right) - 2 \e
\\
&\ge 1 - 3 \e.
\end{align*}
The conclusion follows since $\e$ is arbitrary.
\end{proof}

\section{Artificial Intelligence Usage Statement}

This manuscript is written by the author in collaboration with GPT Sol 5.6 and GPT Astra through the Codex harness as a tool to develop the technical aspects of the proofs under direction from the author. In particular, the research directions and the main proof techniques are mostly contributions of the author or built heavily upon previous literature. The manuscript exposition is almost entirely contributed by the author aside from proofreading and assisting in Tikz code generation.

\appendix

\section{Proof of Theorem~\ref{theorem:laplace_cutoff}}

We first restrict the analysis to a certain set based on the assumptions on the data.
Define the arithmetic mean of the Hessian of the negative log likelihood as
\begin{align*}
\widehat{J}_N(\beta) &= \frac1N\sum_{i=1}^N \sigma'(x_i^T\beta)x_ix_i^T.
\end{align*}
Let $X_N=(x_1,\ldots,x_N)^T$ be the design matrix with rows $x_i^T$,
and let $c>0$ be the constant supplied by
Lemma~\ref{lemma:laplace_curvature}. 
Let $\norm{}_{\mathrm{op}}$ denote the operator norm.
Define the set
\begin{align}
E_N
&=\left\{\norm{X_N^TX_N/N}_{\mathrm{op}}\le2, \norm{\nabla U_N(\beta^*_{d})}\le cN/4, \inf_{\norm\beta\le B+1}
\lambda_{\min}\!\left(\widehat{J}_N(\beta)\right)\ge c \right\}.
\label{eq:E_N}
\end{align}
We will show that $\Prob_*(E_N) \to 1$.

The assumptions on the features give $\E_*(x_ix_i^T)=I_d$ and for every $h=(h_1,\ldots,h_d)^T$ with $\norm{h} = 1$, the centered projection
$x_i^Th-h_1$ has uniformly bounded sub-Gaussian norm by
\citet[Proposition~2.6.1]{vershynin:2018}.
A deterministic intercept $h_1$, with $|h_1|\le1$,
preserves this property.
By \citet[Theorem~5.39]{vershynin:2012}, since $d/N\to0$, there is a constant $b>0$ such that, for all sufficiently large $N \in \Z_+$,
\[
\Prob_*\!\left(\norm{ X_N^T X_N / N }_{\mathrm{op}}>2\right)
\le2e^{-bN}\to 0.
\]

By definition, the gradient
\[
\nabla U_N(\beta^*_{d})=\sum_{i=1}^N
\bigl(\sigma(x_i^T\beta^*_{d})-Y_i\bigr)x_i
+\tau^{-2}\beta^*_{d}.
\]
Using $\E_*\left(\norm{x_i}^2\right) = d$ and $\sigma'\le1/4$ gives
\[
\E_*\left( \norm{\nabla U_N(\beta^*_{d})}^2 \right)
\le\frac{Nd}{4}+\tau^{-4}B^2.
\]
For any fixed $c>0$, Markov's inequality implies
\[
\Prob_*\!\left(\norm{\nabla U_N(\beta^*_{d})}>cN/4\right)
\le\frac{16}{c^2N^2}\left(\frac{Nd}{4}+\tau^{-4}B^2\right)
\to 0.
\]
The following lemma handles the more difficult set to control.

\begin{lemma}
\label{lemma:laplace_curvature}
Under the assumptions on the data $(x_i, Y_i)_{i = 1}^N$ and $d/N\to0$, there is a constant
$c>0$, depending only on $B$ and $K$, such that

\begin{align*}
\Prob_*\left( \inf_{\norm\beta\le B+1}
\lambda_{\min}\!\left(\widehat{J}_N(\beta)\right)\ge c\right) \to 1.
\end{align*}
\end{lemma}

\begin{proof}
Let $\beta \in \R^d$.
The smallest eigenvalue satisfies
\[
\lambda_{\min}\!\left(\widehat J_N(\beta)\right)
=\inf_{\norm h=1}h^T\widehat J_N(\beta)h.
\]
For any $L > 0$, $\sigma'(t)\ge\sigma'(L)$ for $|t|\le L$.
For any $\norm{h} = 1$, the Cauchy-Schwarz inequality gives
\begin{align*}
h^T\widehat J_N(\beta)h
= \frac1N\sum_{i=1}^N\sigma'(x_i^T\beta)(x_i^Th)^2 
&\ge\sigma'(L)\left[
\frac1N\sum_{i=1}^N|x_i^Th|
\mathbf1_{\{|x_i^T\beta|\le L\}}\right]^2
\\
&\ge \sigma'(L)\left[
\frac1N\sum_{i=1}^N|x_i^Th|
-\frac1N\sum_{i=1}^N|x_i^Th|\mathbf1_{\{|x_i^T\beta|>L\}}
\right]^2.
\end{align*}

By \citet[Theorem~5.39]{vershynin:2012}, there is a constant $b>0$ such that, for all
sufficiently large $N$,
\[
\Prob_*\!\left(\norm{ X_N^T X_N / N }_{\mathrm{op}}>2\right)
\le 2 e^{-bN}.
\]
For every $\norm\beta\le B + 1$ and $\norm h=1$, the Cauchy-Schwarz inequality and
$\mathbf1_{\{|x_i^T\beta|>L\}}\le(x_i^T\beta)^2/L^2$ imply on $\{\norm{X_N^TX_N/N}_{\mathrm{op}}\le2\}$,
\begin{align*}
\frac1N\sum_{i=1}^N|x_i^Th|
\mathbf1_{\{|x_i^T\beta|>L\}}
&\le\left(\frac1N\sum_{i=1}^N(x_i^Th)^2\right)^{1/2}
\left(\frac1N\sum_{i=1}^N\mathbf1_{\{|x_i^T\beta|>L\}}\right)^{1/2}\\
&\le\frac1L\left(\frac1N\sum_{i=1}^N(x_i^Th)^2\right)^{1/2}
\left(\frac1N\sum_{i=1}^N(x_i^T\beta)^2\right)^{1/2}\\
&=\frac{\sqrt{(h^TX_N^TX_Nh)(\beta^TX_N^TX_N\beta)}}{NL}
\le\frac{2(B + 1)}{L}.
\end{align*}

By \citet[Proposition~2.5.2]{vershynin:2018}, there is a constant $M > 0$,
depending only on $K$, such that
\begin{align}
\sup_{1\le i\le N}\sup_{\norm h=1}\E_*\left(|x_i^Th|^4\right)\le M.
\label{eq:x_moment_ub}
\end{align}
For every $\norm{h} = 1$, H\"older's inequality with exponents $3/2$ and $3$ gives
\begin{align*}
\E_*(|x_i^Th|^2)
&= \E_*(|x_i^Th|^{2/3}|x_i^Th|^{4/3})\\
&\le \E_*( |x_i^Th| )^{2/3} \E_*( |x_i^Th|^{4} )^{1/3}.
\end{align*}
Then
\[
\E_*(|x_i^Th|)
\ge\frac{\{\E_*\left((x_i^Th)^2\right)\}^{3/2}}
{\{\E_*\left(|x_i^Th|^4\right)\}^{1/2}}
\ge\frac1{\sqrt{M}}.
\]

Taking the infimum
\begin{align*}
\inf_{\norm h=1} \frac1N\sum_{i=1}^N|x_i^Th|
&\ge \inf_{\norm h=1} \frac1N\sum_{i=1}^N\E_*(|x_i^Th|)
+ \inf_{\norm h=1} \left\{ \frac1N\sum_{i=1}^N\bigl(|x_i^Th|-\E_*(|x_i^Th|)\bigr) \right\}
\\
&\ge \frac1{\sqrt M}
- \sup_{\norm h\le1} \left|\frac1N\sum_{i=1}^N
\bigl(|x_i^Th|-\E_*(|x_i^Th|)\bigr)\right|.
\end{align*}

Let $\e = (\e_1,\ldots,\e_N)$ be a vector of independent Rademacher random variables,
independent of the data, each taking the values $1$ and $-1$ with
probability $1/2$. Write $\E_{\e}$ for expectation over these random variables.
Symmetrization and the Cauchy-Schwarz inequality \citep[see Exercises~6.4.4(a) and~6.7.7]{vershynin:2018} yield
\begin{align*}
\E_*\left(\sup_{\norm h\le1}\left|\frac1N\sum_{i=1}^N
\bigl(|x_i^Th|-\E_*(|x_i^Th|)\bigr)\right|\right)
&\le\frac4N\E_*\left(\E_\e\left(
\left\lVert\sum_{i=1}^N\e_i x_i\right\rVert\right)\right)\\
&\le\frac4N\left(\E_*\left(\E_\e\left(
\left\lVert\sum_{i=1}^N\e_i x_i\right\rVert^2\right)\right)\right)^{1/2}\\
&=\frac4N\left(\sum_{i=1}^N \E_*\left(\norm{x_i}^2 \right)\right)^{1/2}\\
&= 4\sqrt{\frac dN}.
\end{align*}
The first equality uses $\E_\e(\e_i\e_j)=0$ for $i\ne j$ and
$\E_\e(\e_i^2)=1$; the last uses $\E_*(\norm{x_i}^2)=d$.
By Markov's inequality,
\begin{align*}
\Prob_*\left(\inf_{\norm h=1}\frac1N\sum_{i=1}^N|x_i^Th|
<\frac1{2\sqrt M}\right)
&\le \Prob_*\left(\sup_{\norm h\le1}
\left|\frac1N\sum_{i=1}^N\bigl(|x_i^Th|-\E_*(|x_i^Th|)\bigr)\right|
>\frac1{2\sqrt M}\right)
\\
&\le 8\sqrt{M}\sqrt{\frac dN}.
\end{align*}
Combining these,  for any $\e \in (0, 1)$, choose $N$ large enough so that

\begin{align*}
&\Prob_*\left( \norm{X_N^TX_N/N}_{\mathrm{op}}\le2,
\inf_{\norm h=1}\frac1N\sum_{i=1}^N|x_i^Th|
\ge\frac1{2\sqrt{M}}
\right)\\
&\qquad\ge 1-2e^{-bN}-8\sqrt{M}\sqrt{d/N}
\ge 1-\e.
\end{align*}

Choose $L=8 (B + 1) \sqrt M$ to get with probability at least $1 - \e$,
simultaneously for every $\norm\beta\le B + 1$ and $\norm h=1$,
\begin{align*}
\frac1N\sum_{i=1}^N|x_i^Th|\mathbf1_{\{|x_i^T\beta|\le L\}}
&=\frac1N\sum_{i=1}^N|x_i^Th|
-\frac1N\sum_{i=1}^N|x_i^Th|\mathbf1_{\{|x_i^T\beta|>L\}}\\
&\ge\frac1{2\sqrt M}-\frac{2(B+1)}{L}
=\frac1{4\sqrt M}.
\end{align*}

Choosing $c=\sigma'(8(B+1)\sqrt M)/(16M)>0$, then gives, on the same event,
\[
\Prob_*\left( \inf_{\norm\beta\le B+1} \inf_{\norm{h} = 1} h^T\widehat J_N(\beta)h
\ge\sigma'(L)\left(\frac1{4\sqrt M}\right)^2=c
\right) \ge 1 - \e.
\]
Since $\e\in(0,1)$ is arbitrary, the stated conclusion follows.

\end{proof}

By Lemma~\ref{lemma:laplace_curvature} and the previous bounds, if $d/N\to0$,
\begin{align*}
\Prob_*(E_N^c)
&\le\Prob_*\!\left(\norm{X_N^TX_N/N}_{\mathrm{op}}>2\right)
+\Prob_*\!\left(\norm{\nabla U_N(\beta^*_{d})}>cN/4\right)\\
&\quad+\Prob_*\!\left(\inf_{\norm\beta\le B+1}
\lambda_{\min}\!\left(\widehat J_N(\beta)\right)<c\right)
\to0
\end{align*}
and so $\Prob_*(E_N) \to 1$.
We get the following consequences on $E_N$. For every $\norm\beta\le B+1$,
\[
\nabla^2U_N(\beta)
=N\widehat J_N(\beta)+\tau^{-2}I_d\succeq cNI_d.
\]
For $\norm{\beta-\beta^*_{d}}=1/2$, integrating this Hessian bound gives
\[
(\beta-\beta^*_{d})^T\nabla U_N(\beta)
\ge\frac{cN}{4}-\frac12\norm{\nabla U_N(\beta^*_{d})}
\ge\frac{cN}{8}>0.
\]
Convexity therefore places the minimizer inside this sphere, and hence
\begin{align}
\norm{\widehat\beta_N-\beta^*_{d}}
&\le\frac12,
\qquad cNI_d\preceq\widehat H_N\preceq CNI_d.
\label{eq:dense_logistic_mode}
\end{align}
The upper Hessian bound follows from
$\nabla^2U_N\preceq X_N^TX_N/4+\tau^{-2}I_d$.

We use $C>0$ for fixed constants that may change between different
inequalities.
Define $s_{N,i}=\widehat H_N^{-1/2}x_i$ and let $S_N$ be the matrix with rows $s_{N,i}^T$.
Then the following bounds hold on $E_N$:
\begin{align}
\norm{S_N^TS_N}_{\mathrm{op}}&\le C,\qquad
\max_{1\le i\le N}\norm{s_{N,i}}^2\le C\frac dN,
\nonumber\\
\sum_{i=1}^N\norm{s_{N,i}}^2&\le Cd,\qquad
\norm{v_N}^2\le C\frac{d^2}{N}.
\label{eq:dense_logistic_bounds}
\end{align}
Indeed, $S_N^TS_N=\widehat H_N^{-1/2}X_N^TX_N\widehat H_N^{-1/2}$
gives the first bound, and
$\norm{x_i}^2\le(1+K^2)d$ gives the second.
Taking the trace gives the third.
We have the identity
\[
v_N = \frac12 \sum_{i=1}^N
\sigma''(x_i^T\widehat\beta_N) \norm{s_{N,i}}^2 s_{N,i}.
\]
For the last, the definition of $v_N$ and boundedness of $\sigma''$ give
\[
\norm{v_N}^2
=\frac14\left\lVert S_N^T
\bigl(\sigma''(x_i^T\widehat\beta_N)\norm{s_{N,i}}^2\bigr)_{i=1}^N
\right\rVert^2
\le C\sum_{i=1}^N\norm{s_{N,i}}^4
\le C\frac{d^2}{N}.
\]

Define the importance weight $w_N = d\Pi_N / d\Q_N$.
Define the Gaussian measure
\begin{align}
\mathcal{G_N}
= N\left(\widehat{\beta}_N-  \widehat H_N^{-1/2}v_N,
          \widehat H_N^{-1}\right).
\label{eq:G_N}
\end{align}
Then we have the decomposition of the importance weight
\[
\log(w_N)
=\log\left( \frac{d\mathcal G_N}{d\Q_N} \right) + \log\left( \frac{d\Pi_N}{d\mathcal G_N} \right).
\]
We will prove in Lemma~\ref{lemma:laplace_tv} that $\norm{\Pi_N-\mathcal G_N}_{TV}\to0$ in $\Prob_*$ probability.
For $Z^0_N \sim N(0,I_d)$ conditional on the data, a comparison of Gaussian densities gives
\[
\log\left( \frac{d\mathcal G_N}{d\Q_N} \right)
\left( \widehat{\beta}_N + \widehat{H}_N^{-1/2}(Z^0_N - v_N) \right)
=\frac{\norm{v_N}^2}{2}-v_N^T Z^0_{N}.
\]
Thus, conditional on any data set with $\norm{v_N}>0$, under $\mathcal G_N$,
\[
\frac{\log\left(d\mathcal G_N/d\Q_N\right)-\norm{v_N}^2/2}
{\norm{v_N}}
\sim N(0, 1).
\]
Total variation convergence implies for every such data and every $s\in\R$,
\[
\left|\Pi_N\!\left(
\frac{\log\left(d\mathcal G_N/d\Q_N\right)-\norm{v_N}^2/2}
{\norm{v_N}}\le s\right)-\Phi(s)\right|
\le\norm{\Pi_N-\mathcal G_N}_{TV}.
\]
From any subsequence, there is a further subsequence along which
\[
\norm{\Pi_N-\mathcal G_N}_{TV}\to0,
\qquad\norm{v_N}\to\infty
\]
hold $\Prob_*$ almost surely.
For every $\e>0$,
\begin{align*}
&(1-e^{-\e})\,
\Pi_N\!\left(\left|\log\left( \frac{d\Pi_N}{d\mathcal G_N} \right) \right|>\e\right)\\
&\qquad\le\int_{\R^d}\left|1-\frac{d\mathcal G_N}{d\Pi_N}\right|\,d\Pi_N
=2\norm{\Pi_N-\mathcal G_N}_{TV}.
\end{align*}
This implies $\log\left(d\Pi_N/d\mathcal G_N\right)\to0$ in $\Pi_N$ probability holds $\Prob_*$ almost surely.

Combining this with the transferred Gaussian limit justifies the following limit
along the extracted subsequence.
Therefore, $\Prob_*$ almost surely, for $s\in\R$,
\[
\Pi_N\!\left(\frac{\log(w_N) - \norm{v_N}^2/2}{\norm{v_N}}\le s\right) \to \Phi(s).
\]
We verify the hypotheses of Lemma~\ref{lemma:general_cutoff_profile} along this further subsequence.
Choose $L_N=\exp(\sqrt{\norm{v_N}})$.
Since $\norm{v_N}\to\infty$, both $T_N(z)$ and $L_N$ tend to infinity,
and the integer part in the definition of $T_N(z)$ satisfies
\[
\log\left(T_N(z)\right)-\frac{\norm{v_N}^2}{2}-z\norm{v_N}\to0.
\]
Also, $\log(L_N)/\norm{v_N}=1/\sqrt{\norm{v_N}}\to0$ and
\[
\frac{\log\left(T_N(z)\right)\pm\log(L_N)-\norm{v_N}^2/2}
{\norm{v_N}}\to z.
\]
Continuity of the Gaussian distribution function $\Phi$ therefore gives
\begin{align*}
&\Pi_N\!\left(\frac{w_N}{T_N(z)}\le\frac1{L_N}\right)
\to\Phi(z),
&\Pi_N\!\left(\frac1{L_N}\le\frac{w_N}{T_N(z)}\le L_N\right)
\to0.
\end{align*}
Lemma~\ref{lemma:general_cutoff_profile} yields
\[
R_{T_N(z)}(\Q_N,\Pi_N)\to2(1-\Phi(z))
\]
along the almost surely convergent data subsequences.
Since every subsequence has such an almost surely convergent subsequence, convergence holds in $\Prob_*$ probability.
It remains to prove Lemma~\ref{lemma:laplace_tv}.

\begin{lemma}
\label{lemma:laplace_tv}
Under the assumptions on the data $(x_i, Y_i)_{i = 1}^N$ and assuming
$
d^{5/3} / N \to 0,
$
then $\norm{\Pi_N-\mathcal G_N}_{TV}\to0$ in $\Prob_*$ probability,
where $\mathcal{G}_N$ is defined by \eqref{eq:G_N}.
\end{lemma}

\begin{proof}
As above, $C>0$ denotes a fixed constant that may change between
inequalities.
The assumption implies
$d^{3/2}/N\le d^{5/3}/N\to0$, and hence also $d/N\to0$.
Define $u_N$ by
\begin{align*}
u_N(z)&=U_N( \widehat\beta_N+\widehat H_N^{-1/2}z )-U_N(\widehat\beta_N).
\end{align*}
We have $u_N(0)=0$ and $\nabla u_N(0)=0$.
Let $P_N$ and $G_N$ be the images of $\Pi_N$ and $\mathcal G_N$
under the invertible affine change of variables
$\beta\mapsto\widehat H_N^{1/2}(\beta-\widehat\beta_N)$.
For every Borel set $A\subseteq\R^d$,
\[
P_N(A)=\Pi_N( \widehat\beta_N+\widehat H_N^{-1/2}A ),
\qquad G_N(A)=\mathcal G_N( \widehat\beta_N+\widehat H_N^{-1/2} A ).
\]
Then $P_N(dz)\propto e^{-u_N(z)}\,dz$ and $G_N=N(-v_N,I_d)$.
Choose $c>0$ to be the constant in Lemma~\ref{lemma:laplace_curvature}.
Choose a fixed $0<\rho\le\sqrt c/4$, set $r_N=\rho\sqrt N$, and write
$B_{r_N}=\{z:\norm z\le r_N\}$. Let $Z_N\sim G_N$ and define
\[
\eta_N^2=\E_{G_N}\left(\norm{\nabla u_N(Z_N)-Z_N-v_N}^2\right).
\]
On $E_N$, \eqref{eq:dense_logistic_mode} gives
$\norm{\widehat H_N^{-1/2}z}\le1/2$ whenever $\norm z\le2r_N$.
For $\norm z\le2r_N$,
\begin{align*}
\norm{\widehat\beta_N+\widehat H_N^{-1/2}z}
\le\norm{\beta^*_{d}}+\norm{\widehat\beta_N-\beta^*_{d}} + \norm{\widehat H_N^{-1/2}z}
\le B+\frac12+\frac12=B+1.
\end{align*}
For $\norm z\le2r_N$, $\nabla^2U_N\succeq cNI_d$ and
$\widehat H_N\preceq CNI_d$ and let $m = \frac{c}{C}$, which is a fixed constant.
Then
\begin{align*}
\nabla^2u_N(z)
=\widehat H_N^{-1/2}
\nabla^2U_N\!\left(\widehat\beta_N+\widehat H_N^{-1/2}z\right)
\widehat H_N^{-1/2}
\succeq cN\widehat H_N^{-1}
\succeq mI_d.
\end{align*}

Since $\nabla^2u_N\succeq mI_d$ on $B_{r_N}$ for a fixed $m>0$ independent of $N$, $d$, and the data, then \citet[Lemma~2.11]{katsevich:2025}, applied with potentials
$u_N(z)$ and $\norm{z+v_N}^2/2$, yields the upper bound
\begin{align}
\norm{\Pi_N-\mathcal G_N}_{TV}
=\norm{P_N-G_N}_{TV}
&\le P_N(B_{r_N}^c)+G_N(B_{r_N}^c)
     +\frac{\eta_N}{2\sqrt{mG_N(B_{r_N})}}.
\label{eq:tv_ub}
\end{align}
On $E_N$, we have the upper bound $\norm{v_N}^2\le Cd^2/N$ in
\eqref{eq:dense_logistic_bounds} and using Markov's
inequality gives
\[
G_N(B_{r_N}^c)
\le\frac{d+\norm{v_N}^2}{r_N^2}
\le\frac{C}{\rho^2}\left(\frac dN+\frac{d^2}{N^2}\right)
\to0.
\]
For each fixed data set, the Gaussian prior ensures Gaussian decay of
$e^{-u_N(z)}$, and $\nabla u_N(z)$ grows at most linearly. Integration
by parts therefore gives
\[
0=\int_{\R^d}\operatorname{div}\!\left(ze^{-u_N(z)}\right)\,dz
=\int_{\R^d}\left(d-z^T\nabla u_N(z)\right)e^{-u_N(z)}\,dz,
\]
and hence
\[
\int_{\R^d}z^T\nabla u_N(z)\,P_N(dz)=d.
\]
Convexity and $\nabla u_N(0)=0$ imply
$z^T\nabla u_N(z)\ge0$ for every $z\in\R^d$, which we use with Markov's inequality.
On $E_N$, local curvature at radius $r_N$ and monotonicity of radial
derivatives give $z^T\nabla u_N(z)\ge mr_N^2$ for $\norm z\ge r_N$.
Thus
\[
P_N(B_{r_N}^c)
\le\frac{1}{mr_N^2}\int_{\R^d}z^T\nabla u_N(z)\,P_N(dz)
=\frac{d}{m\rho^2N}\to0.
\]

It remains to prove on $E_N$ that $\eta_N^2\to0$.
Taking the derivative and using $\nabla U_N(\widehat\beta_N)=0$ gives, for every $z\in\R^d$,
\begin{align*}
\nabla u_N(z)
&=\widehat H_N^{-1/2}
\bigl[\nabla U_N(\widehat\beta_N+\widehat H_N^{-1/2}z)
      -\nabla U_N(\widehat\beta_N)\bigr]\\
&=\sum_{i=1}^N
\bigl[\sigma(x_i^T\widehat\beta_N+s_{N,i}^Tz)
      -\sigma(x_i^T\widehat\beta_N)\bigr]s_{N,i} +\tau^{-2}\widehat H_N^{-1}z
\\
&=\sum_{i=1}^N
\sigma'(x_i^T\widehat\beta_N)(s_{N,i}^Tz) s_{N,i}
+ \frac12 \sum_{i=1}^N \sigma''(x_i^T\widehat\beta_N)(s_{N,i}^Tz)^2 s_{N,i}
+ \tau^{-2}\widehat H_N^{-1}z + r_{N,3}(z)
\end{align*}
where the remainder
\[
r_{N,3}(z)
=\frac12\sum_{i=1}^N(s_{N,i}^Tz)^3s_{N,i}
\int_0^1(1-t)^2\sigma'''(x_i^T\widehat\beta_N+t s_{N,i}^Tz)\,dt.
\]
On $E_N$,
\[
\norm{r_{N,3}(z)}^2
\le\frac{\lVert\sigma'''\rVert_\infty^2}{36}
\norm{S_N^TS_N}_{\mathrm{op}}
\sum_{i=1}^N|s_{N,i}^Tz|^6
\le C\sum_{i=1}^N|s_{N,i}^Tz|^6.
\]
We also have
\begin{align*}
\sum_{i=1}^N\sigma'(x_i^T\widehat\beta_N)(s_{N,i}^Tz)s_{N,i}
&=\widehat H_N^{-1/2}
\left[\sum_{i=1}^N\sigma'(x_i^T\widehat\beta_N)x_ix_i^T\right]
\widehat H_N^{-1/2}z\\
&=\widehat H_N^{-1/2}
\left(\widehat H_N-\tau^{-2}I_d\right)
\widehat H_N^{-1/2}z\\
&=z-\tau^{-2}\widehat H_N^{-1}z.
\end{align*}
Thus, for every $z\in\R^d$,
\begin{align}
\nabla u_N(z)-z
&=\frac12\sum_{i=1}^N\sigma''(x_i^T\widehat\beta_N)
    (s_{N,i}^Tz)^2s_{N,i}+r_{N,3}(z),
\nonumber\\
\norm{r_{N,3}(z)}^2&\le C\sum_{i=1}^N|s_{N,i}^T z|^6.
\label{eq:dense_logistic_gradient}
\end{align}

Write $Z_N=Z^0_{N}-v_N$ with $Z^0_{N}\sim N(0,I_d)$.
The expectations involving $Z^0_{N}$ below are conditional on the observed data, so $s_{N,i}$ and $v_N$ are constants here.
Then expanding the square
\begin{align*}
\nabla u_N(Z_N)-Z_N-v_N
&=\frac12\sum_{i=1}^N\sigma''(x_i^T\widehat\beta_N)s_{N,i}
\left[(s_{N,i}^TZ^0_{N})^2-\norm{s_{N,i}}^2\right]\\
&\quad-\sum_{i=1}^N\sigma''(x_i^T\widehat\beta_N)s_{N,i}
(s_{N,i}^Tv_N)(s_{N,i}^TZ^0_{N})\\
&\quad+\frac12\sum_{i=1}^N\sigma''(x_i^T\widehat\beta_N)s_{N,i}
(s_{N,i}^Tv_N)^2+r_{N,3}(Z_N).
\end{align*}
For the Taylor remainder, \eqref{eq:dense_logistic_gradient} gives
\begin{align*}
\E\left(\norm{r_{N,3}(Z_N)}^2\right)
&\le C\sum_{i=1}^N
\E\left(|s_{N,i}^TZ^0_{N}-s_{N,i}^Tv_N|^6\right)\\
&\le C\sum_{i=1}^N\left[
\E\left(|s_{N,i}^TZ^0_{N}|^6\right)+|s_{N,i}^Tv_N|^6\right]\\
&\le C\sum_{i=1}^N\left[\norm{s_{N,i}}^6+|s_{N,i}^Tv_N|^6\right].
\end{align*}
By \eqref{eq:dense_logistic_bounds},
\begin{align*}
\sum_i\left[(s_{N,i}^Tv_N)^2\norm{s_{N,i}}^2+\norm{s_{N,i}}^6\right]
&\le C\frac{d^3}{N^2},\\
\sum_i(s_{N,i}^Tv_N)^4
&\le C\frac dN\norm{v_N}^4.
\end{align*}

Moreover, $\max_{1\le i\le N}|s_{N,i}^Tv_N|^2\le Cd^3/N^2\to0$,
since $d^3/N^2=(d^{3/2}/N)^2\to0$. Thus, for all sufficiently large
$N$, uniformly on $E_N$,
\[
\sum_{i=1}^N|s_{N,i}^Tv_N|^6
\le\left(\max_{1\le i\le N}|s_{N,i}^Tv_N|^2\right)
\sum_{i=1}^N(s_{N,i}^Tv_N)^4
\le\sum_{i=1}^N(s_{N,i}^Tv_N)^4.
\]

Using the boundedness of $\sigma''$
and the bounds \eqref{eq:dense_logistic_bounds},
\begin{align*}
&\E\left(\left\lVert\sum_{i=1}^N
\sigma''(x_i^T\widehat\beta_N)s_{N,i}
(s_{N,i}^Tv_N)(s_{N,i}^TZ^0_{N})\right\rVert^2\right)\\
&\qquad\le C\sum_{i=1}^N(s_{N,i}^Tv_N)^2
\E\left((s_{N,i}^TZ^0_{N})^2\right)\\
&\qquad=C\sum_{i=1}^N(s_{N,i}^Tv_N)^2\norm{s_{N,i}}^2.
\end{align*}
The same bounds control the constant term in the decomposition
\[
\left\lVert\frac12\sum_{i=1}^N
\sigma''(x_i^T\widehat\beta_N)s_{N,i}(s_{N,i}^Tv_N)^2
\right\rVert^2
\le C\sum_{i=1}^N(s_{N,i}^Tv_N)^4.
\]

For each $i, j$,
\begin{align*}
&\E\left(
\bigl[(s_{N,i}^TZ^0_{N})^2-\norm{s_{N,i}}^2\bigr]
\bigl[(s_{N,j}^TZ^0_{N})^2-\norm{s_{N,j}}^2\bigr]\right)
\\
&=\E\left((s_{N,i}^TZ^0_{N})^2(s_{N,j}^TZ^0_{N})^2\right)
-\norm{s_{N,i}}^2\norm{s_{N,j}}^2\\
&=2(s_{N,i}^Ts_{N,j})^2.
\end{align*}
This implies the identity
\begin{align*}
&\E\left(\left\lVert\frac12\sum_{i=1}^N
\sigma''(x_i^T\widehat\beta_N)s_{N,i}
\bigl[(s_{N,i}^TZ^0_{N})^2-\norm{s_{N,i}}^2\bigr]
\right\rVert^2\right)\\
&\quad=\frac14\sum_{i,j=1}^N
\sigma''(x_i^T\widehat\beta_N)\sigma''(x_j^T\widehat\beta_N)
(s_{N,i}^Ts_{N,j})\\
&\qquad\times\E\left(
\bigl[(s_{N,i}^TZ^0_{N})^2-\norm{s_{N,i}}^2\bigr]
\bigl[(s_{N,j}^TZ^0_{N})^2-\norm{s_{N,j}}^2\bigr]\right)\\
&\quad=\frac12\sum_{i,j=1}^N
\sigma''(x_i^T\widehat\beta_N)\sigma''(x_j^T\widehat\beta_N)
(s_{N,i}^Ts_{N,j})^3.
\end{align*}
Since $\widehat H_N\succeq cNI_d$,
$
0\preceq S_NS_N^T\preceq(cN)^{-1}X_NX_N^T,
$
then the Schur product theorem implies
\[
\left[(s_{N,i}^Ts_{N,j})^3\right]_{i,j=1}^N
\preceq\frac1{(cN)^3}
\left[(x_i^Tx_j)^3\right]_{i,j=1}^N.
\]

Combining the preceding estimates gives, on $E_N$ and for all
sufficiently large $N$,
\begin{align}
\eta_N^2
&\le C\left(\frac1{N^3}\sum_{i,j=1}^N|x_i^Tx_j|^3
            +\frac{d^3}{N^2}+\frac dN\norm{v_N}^4\right).
\end{align}
For $i\ne j$, row independence allows us to use
\eqref{eq:x_moment_ub} conditionally on $x_i$.
H\"older's inequality with exponents $4/3$ and $4$ then yields
\begin{align*}
\E_*\!\left(|x_i^Tx_j|^3\mid x_i\right)
&\le\left\{\E_*\!\left(|x_i^Tx_j|^4\mid x_i\right)\right\}^{3/4}
\\
&\le M^{3/4}\norm{x_i}^3
\\
&\le C d^{3/2}.
\end{align*}
For the diagonal $i=j$, $|x_i^Tx_i|^3=\norm{x_i}^6\le(1+K^2)^3d^3$.
These imply the upper bound
\begin{align*}
\E_*\!\left(\frac1{N^3}\sum_{i,j=1}^N|x_i^Tx_j|^3\right)
&\le\frac{C}{N^3}\left(N(N-1)d^{3/2}+Nd^3\right)\\
&\le C\left(\frac{d^{3/2}}N+\frac{d^3}{N^2}\right)\to0.
\end{align*}
By Markov's inequality, the first term in the bound for $\eta_N^2$
tends to zero in $\Prob_*$ probability. The second term tends to zero
since $d^3/N^2=(d^{3/2}/N)^2\to0$.

For the third term,
\eqref{eq:dense_logistic_bounds} gives, uniformly on $E_N$,
\[
\frac dN\norm{v_N}^4
\le C\frac{d^5}{N^3}
=C\left(\frac{d^{5/3}}N\right)^3\to0.
\]

Since $\Prob_*(E_N)\to1$, it follows that
$\eta_N^2\to0$ in $\Prob_*$ probability. 
The two tail bounds also tend to zero uniformly on $E_N$, so $G_N(B_{r_N})\to1$ in $\Prob_*$ probability. Substituting these limits and $\eta_N^2\to0$ into \eqref{eq:tv_ub} proves $\norm{\Pi_N-\mathcal G_N}_{TV}\to0$ in $\Prob_*$ probability.

\end{proof}

\bibliography{references}

\end{document}